\documentclass[11pt]{article}
\usepackage[english]{babel}
\usepackage{epsfig}
\usepackage{amsmath}
\usepackage{amsthm}
\usepackage{amssymb}
\usepackage{amssymb,color}

\newtheorem{theorem}{Theorem}[section]
\newtheorem{definition}[theorem]{Definition}

\newtheorem{lemma}[theorem]{Lemma}
\newtheorem{proposition}[theorem]{Proposition}
\newtheorem{corollary}[theorem]{Corollary}
\newtheorem{remark}[theorem]{Remark}
\newtheorem{example}[theorem]{Example}

\begin{document}

\title{\bf Octonionic monogenic and slice monogenic Hardy and Bergman spaces}

\author{
Fabrizio Colombo\\
Dipartamento di Matematica\\
Politecnico di Milano\\
Via Bonardi 9\\
20133 Milano, Italy\\
fabrizio.colombo@polimi.it\\
\and
Rolf S\"{o}ren Krau{\ss}har\\
Chair of Mathematics\\
Faculty of Education\\
University of Erfurt\\
Nordh\"auser Str. 63\\
99089 Erfurt, Germany\\
soeren.krausshar@uni-erfurt.de\\
\and
Irene Sabadini\\
Dipartamento di Matematica\\
Politecnico di Milano\\
Via Bonardi 9\\
20133 Milano, Italy\\
irene.sabadini@polimi.it}

\maketitle
\begin{abstract}
In this paper we discuss some basic properties of octonionic Bergman and Hardy spaces. In the first part we review some fundamental concepts of the general theory of octonionic Hardy and Bergman spaces together with related reproducing kernel functions in the monogenic setting.
We explain how some of the fundamental problems in well-defining a reproducing kernel can be overcome in the non-associative setting by looking at the real part of an appropriately defined para-linear octonion-valued inner product. The presence of a weight factor of norm $1$ in the definition of the inner product is an intrinsic new ingredient in the octonionic setting.
Then we look at the slice monogenic octonionic setting using the classical complex book structure. We present explicit formulas for the slice monogenic reproducing kernels for the unit ball, the right octonionic half-space and strip domains bounded in the real direction. In the setting of the unit ball we present an explicit sequential characterization which can be obtained by applying the special Taylor series representation of the slice monogenic setting together with particular octonionic calculation rules that reflect the property of octonionic para-linearity.
\end{abstract}


\noindent {\bf Mathematical Review Classification numbers}: 30G35, 17D05

\section{Introduction}

During the last decades an enormous effort has been done in generalizing classical complex analysis to higher dimensions in many different ways. Apart from the classical approach of considering holomorphic functions in several complex variables, which is a commutative setting, there are numerous different approaches offering other function theories with a range of values in non-commutative and even non-associative algebras. In this paper, we concentrate ourselves on the eight-dimensional non-associative octonionic setting. Following the old works of Hurwitz, the Cayley octonions form the largest normed division algebra over the real numbers and thus represent a very special important case. In this context the term algebra has to be understood in the wider sense admitting non-associativity.
Even this particular setting offers several different approaches of generalizing complex function theory. This paper focusses on two function classes. The first part of the paper addresses the set of octonionic monogenic functions which are octonion-valued functions satisfying the first order generalized Cauchy-Riemann system $ \sum_{i=0}^7 e_i \frac{\partial }{\partial x_i} f$ where $e_0=1$ is the multiplicative neutral element and where the elements $e_i$, $i=1,\ldots, 7$ denote the  linear independent imaginary octonionic units. For the basic theory, see for example \cite{CoSaStr_book,DS,Nono,XL2001,XL2002,KraAACA}. In the second part of the paper we look at the set of octonionic slice monogenic functions in the sense of using the classical complex book structure, see e.g. \cite{GP,GS,RY}.
\par\medskip\par
In contrast to the associative setting of Clifford analysis (cf. for instance \cite{BDS}),  octonionic monogenic and slice monogenic  functions are only endowed with the algebraic structure of an $\mathbb{R}$-module. They do not form an $\mathbb{O}$-module. This seems to represent a serious obstacle in the development of a theory of octonionic reproducing kernel Hilbert modules, see for instance \cite{ConKra2021,Prather2021,QR2021,QR2022}. Note that all the classical theorems from functional analysis explicitly use the Cauchy-Schwarz inequality in their standard proofs.

Since we do not have a direct analogue of such an inequality one has to be extremely careful in the consideration of inner products. Notice that even the most fundamental theorems like the Riesz representation theorem or the existence of an adjoint operator all rely on the Cauchy-Schwarz inequality. The latter however does not hold for granted in the context of octonion-valued inner products.

This lack needs to be carefully taken into account when we want to introduce meaningful generalizations of octonionic monogenic Hardy and Bergman modules.  One effective possibility to overcome this serious problem is to work with real-valued inner products on the sets of $L^2$-integrable monogenic functions such as proposed in \cite{ConKra2021}. In fact, as explained in \cite{QR2021,QR2022}, instead of requiring that an octonion valued functional ${\cal{T}}$ say ${\cal{T}}(f):=(f,g)$ should be $\mathbb{O}$-linear in the classical sense that ${\cal{T}}(f \alpha) = ({\cal{T}}(f)) \alpha$ for all $\alpha \in \mathbb{O}$, the adequate condition in the octonionic setting is the so-called $\mathbb{O}$-para-linearity just demanding that ${\rm Re}[\alpha,f,{\cal{T}}] =0$ where $[\alpha,f,{\cal{T}}]:={\cal{T}}(f \alpha)-({\cal{T}}(f))\alpha$ is the second associator.  Equivalently this condition can be re-interpreted in the form ${\rm Re}(f\alpha,g) = {\rm Re}((f,g)\alpha)$, so we can work on the level of real-valued inner products.
\par\medskip\par
The deviation over the real inner product allows us to apply the corresponding standard functional analytic results on the real components of the octonionic functions. Then the deal consists in finding a suitable way how to lift all the component functions to one  octonionic function. This is a highly non-trivial problem and it has not yet been solved completely by now. In Section 3.1 we first carefully describe the theoretical background of Hardy spaces of octonionic monogenic functions. We also summarize very concisely some explicit formulas both for the Szeg\"o and Bergman kernels of some special domains, cf. also \cite{KraMMAS}. In this case we have a global lifting to one explicitly given octonionic function. We explain how general domains can be addressed and how a reproducing kernel can be obtained.

An important novelty of this paper consists in also providing the fundamental theoretical background for the treatment of the Bergman case. This is presented in detail in Section 3.2. Notice that in \cite{ConKra2021} we exclusively restricted ourselves to the Hardy space case. Here, we now extend these ideas and explain how also the Bergman case can be treated for rather general domains. To do so we have to implement a new idea, namely involving special weight functions of norm $1$. In fact, the  implementation of special weight factors of norm $1$ actually appear as an intrinsic feature in the definition of Bergman an Hardy spaces over non-associative algebras.

The cases of the unit ball and the half-space considered in \cite{WL2018,WL2020} then naturally fit within the more general definitions of Bergman spaces that we propose in Section 3.2.

\par\medskip\par

After having discussed this topic in the monogenic context, in Section 4 we turn to the slice monogenic context and explain how the basic definitions for Hardy and Bergman space case can also be introduced there.

Also in this setting we need to take the aspect of $\mathbb{O}$-para-linearity carefully into account when defining the appropriate analogues of the inner products.

The use of weight factors of norm $1$ inside the definition of the inner products turns out to be a fundamental and intrinsic ingredient in the setting of slice-monogenic Hardy spaces, too. However, in the slice-monogenic Bergman case these weight factors then are canceled out. The reason is that in the slice-monogenic Bergman case we deal on the one hand with a scalar-valued differential form and on the other hand with a particular definition of inner product that is a complex-valued expression on each slice. Therefore, we find ourselves in a special context that allows us to use Artin's theorem and these weight factors naturally disappear as in the associative case.

We prove explicit representation formulas for the slice monogenic Bergman and Szeg\"o kernel for the unit ball, the right octonionic half-space and for strip domains bounded in the direction of the real axis.
In the special case of the unit ball, we present a sequential characterization of functions belonging to the slice monogenic Hardy and Bergman space in terms of their Taylor coefficients. To do so we exploit that slice monogenic functions possess a particularly simple convergent Taylor series representation of the form $\sum\limits_{n\ge 0} x^n a_n$ similarly to the complex case. This leads to an equivalent way of describing slice octonionic Bergman and Hardy spaces over the unit ball either in a geometric way or in a sequential way just encoding the information of the Taylor coefficients. The octonionic Moufang rules provide the key ingredient to set up this characterization which in turn relies on the $\mathbb{O}$-para-linearity of the properly chosen.
\par\medskip\par
Summarizing, this paper hence provides the basic fundament for a subsequent study of Bergman and Hardy spaces in the octonionic monogenic and slice monogenic setting.

\section{Preliminaries}
\subsection{Basics on octonions}

For some basic information on the octonions we recommend to read for instance the overview paper \cite{Baez} which summarizes important elementary properties.

The octonions denoted by $\mathbb{O}$, also called Cayley numbers, are hypercomplex numbers of the form
$$
x = x_0 + x_1 e_1 + x_2 e_2 + x_3 e_3 + x_4 e_4 + x_5 e_5 + x_6 e_6 + x_7 e_7.
$$
As a real-vector space $\mathbb{O}$ is isomorphic to $\mathbb{R}^8$ and $x_0=:{\rm Re}(x)$ is called the real part of $x$. The imaginary part of $x$ denoted by $\Im(x)$ is an element of the seven-dimensional subspace $span_{\mathbb{R}}\{e_1,\ldots,e_7\}$.

The addition of two octonions and the multiplication of an octonion with a real number is defined in the same way as the standard addition and the scalar multiplication in  $\mathbb{R}^8$. However, octonions also can be endowed with a closed multiplication operation. This can be defined on the imaginary units by putting $e_4=e_1 e_2$, $e_5=e_1 e_3$, $e_6= e_2 e_3$ and $e_7 = e_4 e_3 = (e_1 e_2) e_3$ following here and throughout the whole paper the particular notation used in \cite{Baez}.
To leave it simple we also write $e_0:=1$ for the neutral element.
The general octonionic multiplication follows by the table:
\begin{center}
	\begin{tabular}{|l|rrrrrrr|}
		$\cdot$ & $e_1$&  $e_2$ & $e_3$ & $e_4$ & $e_5$ & $e_6$  & $e_7$ \\ \hline
		$e_1$  &  $-1$ &  $e_4$ & $e_5$ & $-e_2$ &$-e_3$ & $-e_7$ & $e_6$ \\
		$e_2$ &  $-e_4$&   $-1$ & $e_6$ & $e_1$ & $e_7$ & $-e_3$ & $-e_5$ \\
		$e_3$ &  $-e_5$& $-e_6$ & $-1$  & $-e_7$&$e_1$  & $e_2$  & $e_4$ \\
		$e_4$ &  $e_2$ & $-e_1$ & $e_7$ & $-1$  &$-e_6$ & $e_5$  & $-e_3$\\
		$e_5$ &  $e_3$ & $-e_7$ & $-e_1$&  $e_6$&  $-1$ & $-e_4$ & $e_2$ \\
		$e_6$ &  $e_7$ &  $e_3$ & $-e_2$& $-e_5$& $e_4$ & $-1$   & $-e_1$ \\
		$e_7$ & $-e_6$ &  $e_5$ & $-e_4$& $e_3$ & $-e_2$& $e_1$  & $-1$ \\ \hline 	
	\end{tabular}
\end{center}
The octonionic multiplication is closed, however, it is not associative. Using the above mentioned table it is easy to construct simple counterexamples. Nevertheless, in $\mathbb{O}$ we still have a couple of special calculation rules: We have left alternativity  $x(xy)=(xx)y$, we have right alternativity $(yx)x=y(xx)$ and there is the
flexibility identity stating that $(xy)x=x(yx)$ for all $x,y \in \mathbb{O}$. Further, there are the four Moufang identies: $
z(x(zy))=((zx)z)y$, $x(z(yz))=((xz)y)z$, $(zx)(yz)=(z(xy))z$ and $(zx)(yz)=z((xy)z)$, cf. \cite{Baez,dieckmann}.

$\mathbb{O}$ is free of zero-divisors thus forming a non-associative composition algebra where the norm satisfies the rule $|xy| = |x||y|$ for all $x,y \in \mathbb{O}$. The norm coincides with the Euclidean norm in $\mathbb{R}^8$, i.e. $|x|=(\sum\limits_{i=0}^7 x_i^2)^{1/2}$. It can also be expressed in terms of $|x| = \sqrt{x \overline{x}}$ where $$\overline x := x_0 - x_1 e_1 - x_2 e_2 - x_3 e_3 - x_4 e_4 - x_5 e_5 - x_6 e_6 - x_7 e_7$$ is the conjugated octonion. The Euclidean inner product of $\mathbb{R}^8$ can be written with octonions $x,y$ in the way $\langle x,y \rangle = {\rm Re}(x\overline{y}) = \sum\limits_{i=0}^7 x_i y_i$. It induces $|x|$ viz $\sqrt{\langle x,x\rangle}$.

\subsection{Octonionic monogenic function theory}

There are different ways to generalize complex function theory to the octonionic case. In this paper discuss the Riemann approach treating the class of octonionic monogenic functions \cite{DS,Nono,XL2001,Kauhanen_2} (for \cite{DS}, see also \cite{CoSaStr_book}) and the slice monogenic setting \cite{GS,GP} using the classical complex book structure.

First we recall, cf. e.g. \cite{DS,Nono,XL2001}:
\begin{definition} (octonionic monogenicity). \\
	Let $U \subseteq \mathbb{O}$ be an open set. A real differentiable function $f:U \to \mathbb{O}$ is called left (right) octonionic monogenic if ${\cal{D}} f = 0$
	(resp. $f {\cal{D}} = 0$). Here,
	${\cal{D}}:= \frac{\partial }{\partial x_0} + \sum\limits_{i=1}^7 e_i \frac{\partial }{\partial x_i}$ is the octonionic first order Cauchy-Riemann operator.  If $f$ satisfies $\overline{{\cal{D}}}f = 0$ (resp. $f\overline{\cal{D}} = 0$), then we call $f$ left (right) octonionic anti-monogenic.
\end{definition}
A substantial difference to the associative Clifford analysis setting consists in the fact that octonionic monogenic functions are maps from $\mathbb{O}$($\cong\mathbb R^8$) to $\mathbb{O}$ while Clifford monogenic functions are maps from $\mathbb{R}^8$ to $\mathbb{R}^{128}$. But even more important is the fact that left (right) octonionic monogenic functions do neither form a right nor a left ${\mathbb{O}}$-module, see for instance the very elementary counterexamples presented in \cite{Kauhanen_3}.\\
The lack of an $\mathbb{O}$-modular structure is a direct consequence of the lack of associativity and this causes serious obstacles. One consequence is a lack of a direct analogue of Stokes' formula, cf. \cite{XL2000}. Neither if both ${\cal{D}} f = 0$ and $g {\cal{D}} = 0$,  statements of the form
 	$$
 	\int\limits_{\partial G} g(x) \; (d\sigma(x) f(x)) = 0 \quad {\rm or}\quad
	\int\limits_{\partial G} (g(x) d\sigma(x)) \; f(x) = 0
 	$$
will be true in general.
In this paper we apply the notation $d\sigma(x) = n(x) |d\sigma(x)|$, where $n(x)$ is the outward directed unit normal field at $x$ and $|d\sigma(x)|$ the scalar surface measure of the $7$-dimensional boundary $\partial G$.
However, as described in \cite{XLT2008}, one has a modified Stokes' formula of the following form
$$
\int\limits_{\partial G} g(x) \; (d\sigma(x)  f(x)) = \int\limits_G \Bigg(
g(x)({\cal{D}} f(x)) + (g(x){\cal{D}})f(x)  - \sum\limits_{j=0}^7 [e_j, {\cal{D}}g_j(x),f(x)]
\Bigg) dV(x),
$$
where $dV(x) = dx_0 \wedge dx_1 \wedge \ldots \wedge dx_7$ is the scalar volume measure and
where $[a,b,c] := (ab)c - a(bc)$ is the so-called (first) associator. The (first) associator measures the non-associativity of three elements similarly as the classical commutator measures the non-commutativity.

Despite of these obstacles one has a close analogue of Cauchy's integral formula, cf. \cite{Nono,XL2002}:
\begin{proposition}\label{cauchy1}(Cauchy's integral formula).\\
Let $U \subseteq \mathbb{O}$ be open and $G \subseteq U$ be an $8$-dimensional compact oriented manifold with a strongly Lipschitz boundary $\partial G$. If $f: U \to \mathbb{O}$ is left octonionic monogenic, then for all $x$ in the interior of $G$
$$
f(x)= \frac{3}{\pi^4} \int\limits_{\partial G} q_{\bf 0}(y-x) \Big(d\sigma(y) f(y)\Big),
$$
where $q_{\bf 0}(y-x)=\dfrac{\bar y-\bar x}{|y-x|^8}$.
 \end{proposition}
The way how the brackets are put is important. Putting them in the different way, leads to a different formula
$$
\frac{3}{\pi^4} \int\limits_{\partial G} \Big( q_{\bf 0}(y-x) d\sigma(y)\Big) f(y)  =   f(x)   +   \int\limits_G \sum\limits_{i=0}^7
\Big[q_{\bf 0}(y-x),{\cal{D}}f_i(y),e_i  \Big] dy_0 \cdots dy_7,
$$
where the (first) associator appears again, cf. \cite{XL2002}.

\subsection{Octonionic slice monogenic functions}

The theory of slice monogenic functions, which has been widely studied in the past fifteen years, has also been successfully performed on the octonions. The first paper in which this class of functions was studied is \cite{GS} following the quaternionic case, and to introduce it we need to introduce more notations.
By $\mathbb{S}$ we denote the unit sphere of purely imaginary
octonions, i.e. $$\mathbb{S}=\{x=\sum^7_{k=1} x_ke_k \ \
\textnormal{such that} \ \ \sum^7_{k=1} x^2_k =1\}.$$
We note that
if $I\in\mathbb{S}$, then $I^2=-1$ so the elements
of $\mathbb{S}$ are called imaginary units and for any $I\in\mathbb S$ we can consider the complex plane
$\mathbb{C}_I=\mathbb{R}+\mathbb{R}I$
containing $1$ and $I$.
\\
Given $x\in\mathbb O$ we can write $x=u+Iv$ for a uniquely defined $I\in\mathbb S$ if $x\not\in\mathbb R$. By $[x]$ we denote the sphere associated with $x$ namely
$$
[x]=\{y=u+Jv\ {\rm \ where\ } J\in\mathbb S,\, x=u+Iv\}.
$$

Next, an open set  $\Omega\subseteq\mathbb O$ is said to be axially symmetric if $[x]\subset \Omega$ whenever $x\in\Omega$.
Moreover, $\Omega$ is called slice domain if it is a domain whose intersection with $\mathbb C_I$ is connected for all $I \in \mathbb S$.

\begin{definition}\label{slicemon} Let $\Omega$ be a domain in
$\mathbb{O}$. A real differentiable function $f:\Omega \to
\mathbb{O}$ is said to be (left) slice monogenic if, for every $I \in
\mathbb{S}$, its restriction $f_I$ to $\mathbb{C}_I$ is holomorphic on $\Omega \cap \mathbb{C}_I$, i.e.
$$\frac{1}{2}\left(\frac{\partial}{\partial u}
+I\frac{\partial}{\partial v}\right)f_I(u+Iv)=0,$$ on $\Omega \cap \mathbb{C}_I$.
\end{definition}
We denote the set of slice monogenic functions on $\Omega$ by $\mathcal{SM}(\Omega)$.
One can write the imaginary unit $I\in\mathbb S$ on the right with respect to the function $f$, and in this case we obtain the class of the so-called right slice monogenic functions. The two function theories contain different functions but they are substantially equivalent. We shall refer to slice monogenic functions, for short, when referring to the case of left slice monogenic functions.
\begin{example} Polynomials with coefficients written on the right  are (left) slice monogenic functions.
\end{example}

By writing the components of the functions in terms of four complex valued functions, one obtains a result, called Splitting Lemma, which is helpful for proving the next statement on the power series expansion of a slice monogenic function in a ball of radius $r$ centered at the origin denoted by $B_8(0,r)$.

\begin{lemma}\label{SLemma} If $f$ is a slice monogenic function
on the open set $\Omega\subseteq\mathbb O$, then for every $I_1 \in \mathbb{S}$, we can find
$I_2$ and $I_4$ in $\mathbb{S}$, such that there are four
holomorphic functions $F_1,F_2,G_1,G_2$ from $B_8(0,r)\cap \mathbb{C}_{I_1}$ to
$\mathbb{C}_{I_1}$ such that for any $z=u+vI_1$, it is
$$f_{I_1}(z)=F_1(z)+F_2(z)I_2+(G_1(z)+G_2(z)I_2)I_4.$$
\end{lemma}
This result hints to a very nice feature of slice monogenic functions in a neighborhood of the origin, namely that they admit a power series expansion (with coefficients on the right) in terms of the octonionic variable, see \cite{GS}:
\begin{theorem} The function $f:B_8(0,r) \to \mathbb{O}$ is slice monogenic, if and only if it has a series
expansion of the form
$$f(x)=\sum_{n=0}^\infty x^n \frac{1}{n!} \frac{\partial^n f}{\partial x^n}(0).$$
\end{theorem}

Slice monogenic functions over $\mathbb O$ can be defined also following \cite{GP} in which the authors consider the more general case of functions with values in a real alternative algebra, but for a special subclass of functions, the so-called slice functions. The idea of considering slice functions over octonions (of intrinsic type) goes back to \cite{DS}, see also the translation in \cite{CoSaStr_book}. When needed, to distinguish between the two approaches followed in \cite{GS} and \cite{GP}, respectively, we shall call the corresponding functions strong slice monogenic in the first case and weak slice monogenic in the latter case. Both classes of functions are considered and compared over the octonions in the paper \cite{DRS} where the general case of functions in $n$ variables is treated. We remark that the case of several octonionic variables has also been studied in \cite{RY}.

In both cases, the functions have the property of satisfying the so-called representation formula and so they can be reconstructed by the values on a complex slice $\mathbb C_I$, see \cite{DRS} and \cite{GP,RY} for the context of slice functions.
We do not enter into the technical definitions needed for the statement, because for the particular domains that we shall consider, namely the unit ball, the half-space of octonions with positive real part and for the strip, the hypothesis are satisfied.
\begin{theorem} (Representation Formula)\label{representation}
Let $\Omega$ be an axially symmetric slice domain. Then for every $I,J\in\mathbb S$ the following formula holds:
\[
f(x+Iy)=\frac 12\left[f(x+yJ)+f(x-yJ)\right] +\frac 12 I\left[J(f(x-yJ)-f(x+yJ))\right].
\]
\end{theorem}
The Cauchy formula was proven in \cite{GPR}. As in the classic monogenic case, it is formally inspired by the formula for functions Clifford algebra-valued, but as in the monogenic setting it also requires a careful use of the parentheses:
\begin{theorem}
Let $\Omega$ be an axially symmetric bounded set in $\mathbb O$, $I\in\mathbb S$ and let $\Omega\cap\mathbb C_I$ have a smooth boundary. Let $f$ be slice monogenic in $\Omega$. Then for every $x\in\Omega$:
$$
f(x)=\frac{1}{2\pi}\int_{\partial (\Omega\cap\mathbb C_I)} S_L^{-1}(s,x)\left[I^{-1} ds f(s)\right],
$$
where $S_L^{-1}(s,x)=-(x^2-2{\rm {\rm Re}}(s)x+|s|^2)^{-1}(x-\bar s)$ is the slice monogenic Cauchy kernel.
\end{theorem}
Now we have the necessary toolkit available to study octonionic Hilbert spaces of monogenic and slice monogenic functions in the following two sections which will represent the core-pieces of this paper.

\section{Octonionic Hardy and Bergman spaces in the monogenic setting}

The aim of this section and the following one is to present some basic results on Hardy and Bergman spaces of octonionic monogenic and slice monogenic functions. The basic theory of their analogues in the associative Clifford analysis are described for example in \cite{BDS}.
\par\medskip\par
In all that follows let us suppose that $\Omega \subset \mathbb{O}$ is a simply-connected orientable domain with a strongly Lipschitz boundary, denoted by $\Sigma=\partial \Omega$, where the outward directed unit normal field denoted by $n(y)$ exists at almost every point $y$ of the boundary $\partial \Omega$.
\par\medskip\par
Before we start we recall the following basic definition from \cite{GH}, see also \cite{QR2021}:
\begin{definition}\label{defHilbert} An octonionic Hilbert space $H$ is a left $\mathbb{O}$-module with an octonionic valued inner product $(\cdot,\cdot): H \times H \rightarrow \mathbb{O}$ such that $(H,\langle \cdot,\cdot\rangle_0)$ is a real Hilbert space, where $\langle \cdot,\cdot\rangle_0 :={\rm Re}(\cdot,\cdot)$. The octonion-valued inner product is supposed to satisfy for all $f,g,h \in H$ and for all $\alpha \in \mathbb{O}$ the following rules
\begin{itemize}
\item[(i)] $(f+g,h) = (f,h) + (g+h)$ (additivity)
\item[(ii)] $(g,f) = \overline{(f,g)}$ (Hermitian property)
\item[(iii)] $(f,f) \in \mathbb{R}^{\ge 0}$ and $(f,f)=0$ iff $f=0$ (strict positivity)
\item[(iv)] $(fr,g) = (f,g) r$ for all real $r \in \mathbb{R}$ ($\mathbb{R}$-homogeneity)
\item[(v)] $(f\alpha,f) = (f,f) \alpha$
\item[(vi)] $\langle f \alpha,g \rangle_0 = {\rm Re}\{(f \alpha,g)\} = {\rm Re}\{(f,g)\alpha\}$ ($\mathbb{O}$-para-linearity)
\end{itemize}
\end{definition}
Canonically, this definition leads to the consideration of octonion-valued functionals defined by ${\cal{T}}: H \to \mathbb{O}, {\cal{T}}(f) := (f,g)$. However, as explained in \cite{QR2021}, the requirement of a functional being $\mathbb{O}$-linear in the classical sense, demanding that ${\cal{T}}(f\alpha)= {\cal{T}}(f) \alpha$ is too strong for developing a powerful theory of octonionic Hilbert function spaces. Following \cite{QR2021} and \cite{QR2022} the adequate way is to demand more weakly that a functional is only $\mathbb{O}$-para-linear in the sense of axiom (vi) of the preceding definition.

\subsection{Hardy spaces of octonionic monogenic functions}

In this section, we introduce one notion of a generalized octonionic Hardy space $H^2(\Omega, \mathbb{O})$:
\begin{definition} The octonionic Hardy space $H^2(\Omega, \mathbb{O})$ is the closure of the set of $L^2(\partial \Omega)$-octonion-valued functions that are left octonionic monogenic functions inside of $\Omega$ and have a continuous extension to the boundary $\partial \Omega$.
\end{definition}
We now need a notion of octonion valued inner product so we follow \cite{ConKra2021} and we give the following definition, which is inspired by \cite{WL2018} and \cite{WL2020}:
\begin{definition} Let $\Omega \subset \mathbb{O}$ be a domain as above.
For octonion-valued functions $f,g \in L^2(\partial \Omega)$ one defines the following $\mathbb{R}$-linear $\mathbb{O}$-valued inner product
$$
(f,g)_{\partial \Omega} := \frac{3}{\pi^4} \int\limits_{\partial \Omega}(\overline{n(x)g(x)})\; (n(x)f(x)) |d\sigma(x)|,
$$
where again $n(x)$ denotes the exterior unit normal at a boundary point $x \in \partial \Omega$ and where $|d\sigma(x)|$ stands for the scalar-valued Lebesgue surface element.
\end{definition}

By a direct verification we show
\begin{proposition}\label{criteriaHardy}
The set $(H^2(\Omega, \mathbb{O}),(\cdot,\cdot)_{\partial \Omega})$ is an octonionic Hilbert space in the sense of Definition~\ref{defHilbert}.
\end{proposition}
\begin{proof}
\begin{itemize}
\item[(i)] The additivity follows by the usual additivity property of the integral.
\item[(ii)] As already explictly shown in \cite{ConKra2021} this inner product satisfies the Hermitian property $\overline{(f,g)} = (g,f)$ in terms of the octonionic conjugation.
\item[(iii)] As also explicitly shown in \cite{ConKra2021} we have $(f,f)   = \|f\|^2_{L^2(\partial \Omega)}$ so in particular $(f,f)=0$ iff $f=0$ because $f$ needs to be continuous in particular as an element belonging to $H^2(\Omega)$.
\item[(iv)]  Since each real number commutes with any octonion, this property follows directly.
\item[(v)] This is a consequence of Artin's theorem. The expression under the integral is generated by two octonionic elements, namely by $\alpha$ and by $F(x)$ where $F(x):= n(x) f(x)$.
\item[(vi)] As also shown in \cite{ConKra2021} we have $\langle f,g\rangle_0 = \langle n f, n g \rangle_0$ for any $n$ with $|n|=1$. To prove this we exploited the property $\langle ab,c \rangle = \langle b, \overline{a}c\rangle$ for all $a,b,c \in \mathbb{O}$, see Corollary 3.5 in \cite{XLT2008} regarding the standard Euclidean inner product on $\mathbb{O}$ (defined by $\langle a,b\rangle:= {\rm Re}(a \overline{b}) = \sum_{i=0}^7 a_i b_i$), together with Artin's theorem.

So in fact $\langle f, g \rangle_0$ can be rewritten in the form
$$
\langle f , g \rangle_0 = \int\limits_{\partial \Omega} {\rm Re}(\overline{g(x)} f(x)) |d\sigma(x)|.
$$
Now, in view of the fact that for any $a,b,c\in\mathbb O$
\begin{equation}\label{1.4e}
{\rm Re} ((ab)c)={\rm Re} (a(bc))={\rm Re} (abc),
\end{equation}
(see Proposition 1.4 (e) in \cite{dieckmann}), we have that $$\langle f \alpha, g \rangle_0 = \int\limits_{\partial \Omega} {\rm Re}(\overline{g(x)} (f(x)\alpha)) |d\sigma(x)| = \int\limits_{\partial \Omega} {\rm Re}((\overline{g(x)} f(x))\alpha) |d\sigma(x)|$$
and this is the required para-linearity property.
\end{itemize}
\end{proof}
In complex or Clifford analysis where we have associativity the integrand simplifies to the expression $\overline{g(x)}f(x)$
and so one gets the standard definition of the Hardy space inner product; the inner product then is independent of $n(x)$. The normal field $n(x)$  is naturally canceled out in an associative setting.
\par\medskip\par
In the octonionic context the presence of the normal field $n(x)$ inside these brackets however really makes an important difference. It namely allows us to recognize the Cauchy transform and to use Cauchy's integral formula in the framework that makes use of this inner product. We will see more clearly in the following parts of this paper that the implementation of the normal field inside the definition of the inner product is a substantial intrinsic feature in the non-associative setting and not only in the monogenic context.
\begin{remark}
In the case where $\Omega = B_8(0,1) = \{x \in \mathbb{R}^8 \mid |x| < 1\}$ is the unit ball we have $n(x) = x$, and hence we re-obtain the following particular definition proposed by Wang and Li in \cite{WL2018}
$$
(f,g)_{S_7} = \frac{3}{\pi^4} \int\limits_{\partial B_8(0,1)} (\overline{x \; g(x)}) \; (x \; f(x)) |d\sigma(x)|,
$$
where $S_7 :=\partial B_8(0,1)$ is used as abbreviation for the unit sphere in $\mathbb{O}$.
\par\medskip\par
In case of considering for $\Omega$ the octonionic right half-space $H^{+}(\mathbb{O}) = \{ x \in \mathbb{O} \mid x_0 > 0\}$ the outward directed unit normal field simplifies to the constant function $n(x) = - 1$ and our general definition for the inner product reduces to
$$
(f,g)_{\partial H^+}  = \frac{3}{\pi^4} \int\limits_{\partial H^+} \overline{g(x)} f(x) |d\sigma(x)|
             = \frac{3}{\pi^4} \int\limits_{\mathbb{R}^7} \overline{g(x)} f(x) dx_1 dx_2 \cdots dx_7,
$$
like in the associative setting. The idea of \cite{WL2020} to use the usual inner product in the context of $H^+(\mathbb{O})$ thus fits very well in this framework.
\end{remark}
Once more, we wish to emphasize that in contrast to the Clifford algebra setting, here $(\cdot,\cdot)$ is only $\mathbb{R}$-linear but not $\mathbb{O}$-linear. We only have the $\mathbb{O}$-para-linearity as defined above. In view of the lack of $\mathbb{O}$-linearity we do neither have a Cauchy-Schwarz inequality nor a direct analogue of the Fischer-Riesz representation theorem when using this inner product. So, the existence of a reproducing kernel is a non-trivial issue.  Before we explain how we can overcome this problem let us briefly give for the sake of being self-contained a concise summary on the results that were obtained in the recent period 2018 -- 2021 in which explicit formulas for the reproducing Szeg\"o kernel have been established in the particular contexts of the domains $B_8(0,1)$, $H^{+}(\mathbb{O})$ and for strip domains $0 < x_0 < d$ bounded in the direction of the real axis, cf.\cite{WL2018,KraMMAS,ConKra2021}. In these papers the existence of the kernel has basically been shown by presenting a fully explicit kernel function that then in fact turned out to have the property of reproducing all functions $f \in (H^2,(\cdot,\cdot))$. The proof of the reproduction property was achieved by a  direct application of the octonionic Cauchy integral formula. To this end, the  presence of $n(y)$ in the definition of the inner product was really essential, because otherwise one could not draw the conclusion of the reproduction property by directly using the specific Cauchy formula from \cite{XL2002}.
\par\medskip\par
In \cite{WL2018} the authors proved that
\begin{proposition}
For the unit octonionic ball $B_8(0,1)$ the octonionic Szeg\"o projection is given by
$$
[{\cal{S}} f](y) := (f,S(\cdot,y))_{S_7}
$$
where
$$
S(x,y): =S_{S_7}(x,y) = \frac{1-\overline{x}y}{|1-\overline{x}y|^8}.
$$
\end{proposition}
Next in \cite{KraMMAS} we were able to prove:
\begin{proposition}
\begin{itemize}
\item[(i)]
For the octonionic right half-space  $H^{+}(\mathbb{O})$ we have $[{\cal{S}} f](y) := (f,S(\cdot,y))_{H^{+}}$ with
$$
S(x,y): =S_{H^+}(x,y) = \frac{\overline{x}+y}{|\overline{x}+y|^8}
$$
for all $f \in H^2(H^+(\mathbb{O}))$.
\item[(ii)]
Let $d > 0$ be a positive real number. The octonionic monogenic Szeg\"o kernel of the strip domain $T := \{ x \in \mathbb{O} \mid 0 < x_0 < d\}$ has the form
$$
S(x,y): =S_{T}(x,y) =  \sum\limits_{n=-\infty}^{+\infty} (-1)^n \frac{\overline{x}+y+2dn}{|\overline{x}+y+2dn|^8}.
$$
\end{itemize}
\end{proposition}
In these two papers the existence of the Szeg\"o kernel has been proved directly by simply proposing a particular function and by showing then that this function exactly reproduces every element of $H^2$. We also note that these kernel functions are left octonionic monogenic  in the first variable and anti-monogenic in the second one. Very detailed proofs have also been given in the overview paper \cite{CoKraSa}.
\par\medskip\par
Let us now turn to the context of general domains. Here we are in need of theoretical arguments. Actually in \cite{ConKra2021} we managed to show that $H^2(\Omega, \mathbb{O})$ equipped with $(\cdot,\cdot)$ satisfies a Bergman condition. This could be achieved by only relying on the Cauchy integral formula.

However, due to the absence of a Fischer-Riesz representation theorem we cannot draw any conclusion on the existence of a reproducing kernel function just by means of the Bergman condition in the form as it is proved in \cite{ConKra2021}.
\par\medskip\par
As in \cite{ConKra2021}, one possibility to overcome this obstacle is to consider first the real-valued inner product defined by
$$
\langle f,g \rangle_0 := {\rm Re}\{(f,g)\}.
$$
Using, instead of the global octonion-valued inner product, the real-valued inner products on each real component function of an octonionic function allows us to apply the well-known statements on classical reproducing kernel Hilbert spaces.
\par\medskip\par
\begin{remark}
The property $\langle ab,c \rangle = \langle b, \overline{a}c\rangle$ allows us to meaningfully define for any octonionic $\mathbb{R}$-linear functional ${\cal{T}}$ a uniquely defined adjoint ${\cal{T}}^*$ such that $  \langle {\cal{T}} f,g \rangle_0 =  \langle f,{\cal{T}}^*g\rangle_0$. Alternatively, we can also argue here with the concept of $\mathbb{O}$-para-linearity as proposed in \cite{QR2022}.
\end{remark}
{\bf Fact}: If we endow $H^2(\Omega,\mathbb{O})$ with $ \langle \cdot,\cdot \rangle_0$, then the classical Fischer-Riesz representation theorem implies that there is always a uniquely defined reproducing kernel ${{S_0}}_x: y \mapsto {{S_0}}_x(y):={S_0}(x,y)$ in $H^2(\Omega)$ which we call the octonionic Szeg\"o kernel with respect to the real-valued inner product  $\langle \cdot , \cdot \rangle_0$.
\par\medskip\par
Note that it reproduces the harmonic real part $f_0$ of an octonionic monogenic function $f$ viz
$$
[{\cal{S}}_0 f](x) :=  \langle f,{{S_0}}_x \rangle_0 = f_0(x) \quad \quad \forall f \in H^2(\Omega, \mathbb{O}).
$$
Let us now consider the other real components of $f(x) = f_0(x) + f_1(x) e_1 + \cdots + f_7(x) e_7$ belonging to $H^2(\Omega)$. To leave it simple let us denote by $f_i$ or by $\{f\}_i$ the $i$-th real component of $f$.

If $f \in H^2(\Omega)$, then all the components $f_i$ are harmonic functions from $\Omega\subset\mathbb{O}\cong\mathbb R^8$ to $\mathbb{R}$.
Next we may define for all $i = 1,\ldots,7$ the real-valued inner products
$$
\langle f,g \rangle_i := \{(f,g)\}_i = \frac{3}{\pi^4}\int\limits_{\partial \Omega} \{(\overline{n(y)\;g(y)})\;(n(y)\;f(y))   \}_i |d\sigma(y)|,
$$
For each $i=1,\ldots,7$ we get a unique reproducing kernel ${{S_i}}_x: y \mapsto {{S_i}}_x(y):={{S_i}}(x,y)$ in $H^2(\Omega)$ reproducing the harmonic component $f_i$ of the octonionic monogenic $f \in H^2(\Omega)$, now with respect to the inner product $ \langle \cdot,\cdot \rangle_i$, i.e.:
$$
[{\cal{S}}_i f](x) :=  \langle f,{{S_i}}_x \rangle_i = f_i(x) \quad \quad \forall f \in H^2(\Omega, \mathbb{O}).
$$
Next a total octonionic monogenic Szeg\"o projection can be defined by
\begin{equation}\label{monszego}
[{\cal{S}} f] := \sum\limits_{i=0}^7 \langle f , S_i \rangle_i e_i.
\end{equation}
By construction this projection satisfies $[{\cal{S}} f](x) = \sum\limits_{i=0}^7 f_i(x) e_i = f(x)$ for all $f\in H^2(\Omega)$ and it is $\mathbb{R}$-linear in the classical sense, cf. \cite{ConKra2021}. The projection ${\cal{S}}_0 f$ constitutes its real part.
\par\medskip\par
{\bf Open question}: Investigate whether it is possible to represent the total Szeg\"o projection
${\cal{S}}$ in terms of a single octonionic integral transform of the form
\begin{equation}\label{globalform}
[{\cal{S}}f](x) := \frac{3}{\pi^4} \int\limits_{\partial \Omega}(\overline{n(y) S(x,y)})\; (n(y)f(y)) |d\sigma(y)|
\end{equation}
with a function $S(x,y)$ that is octonionic monogenic in $x$ and octonionic anti-monogenic in $y$.
\par\medskip\par
Positive answers have in fact already been given for the special cases where either $\Omega = B_8(0,1)$, or $\Omega = H^+(\Omega)$, or where $\Omega$ is a strip domain of the above described form by proving the reproduction property directly. But even in the general cases we can always reproduce component-wisely using the real-valued inner products.

\par\medskip\par
{\bf Problem}: Due to the absence of a direct analogue of the Fischer-Riesz representation theorem we cannot directly conclude that one can always express the projection ${\cal{S}}$ in terms of one single octonionic integral kernel $S(x,y)$ being octonionic monogenic in $x$ and octonionic anti-monogenic in $y$ of one single octonionic integral equation.
\par\medskip\par
{\bf Fact A}: The consideration of the real-valued inner products $\langle \cdot, \cdot \rangle_i$ links the Szeg\"o projection with the notion of orthogonality and self-adjointness with respect to the real-valued inner products $\langle \cdot, \cdot \rangle_i$.


Using the real-valued inner product  $\langle f , g \rangle_0$ we have the self-adjointness property of an $\mathbb{R}$-linear octonionic integral operator ${\cal{T}}$ if  $\langle {\cal{T}} f , g \rangle_0 = \langle f , {\cal{T}} g \rangle_0$.
\par\medskip\par

{\bf Fact B}: Assume that ${\cal{T}}f$ is representable with a single octonionic kernel function $k(x,y)$ being octonionic monogenic in $x$ and octonionic anti-monogenic in $y$. This operator is self-adjoint with respect to $\langle \cdot, \cdot \rangle_0$  if the associated kernel function of ${\cal{T}}^*$ is exactly the octonionic conjugate of the kernel function of ${\cal{T}}$ in view of compatibility condition $\langle u v,w\rangle = \langle v , \overline{u}w\rangle$.

\par\medskip\par
It has not yet been proven that the total Szeg\"o projection is always an orthogonal projection and self-adjoint with respect to $\langle\cdot,\cdot\rangle_0$. However, in all the cases where we have an explicit representation formula we can directly establish the following:
\begin{proposition}
In the explicit cases of $B_8(0,1)$, $H^+(\mathbb{O})$ and the strip domains $T = \{z \in \mathbb{O} \mid 0 < x_0 < d\}$ the total Szeg\"o projection is self-adjoint and hence an orthogonal projector with respect to the real-valued inner product $\langle\cdot, \cdot\rangle_0$.
\end{proposition}
\begin{proof}
We want to show that in the above mentioned three cases the relation $\langle {\cal{S}}f,g\rangle_0 = \langle f, {\cal{S}}g\rangle_0$ holds.  To do this we can directly verify in each case that  actually $\overline{S(x,y)} = S(y,x)$, use the argument of Fact B. Concretely:
\begin{itemize}
\item
In the case of the octonionic unit ball we have
$$
\overline{S_{S^7}(x,y)}= \overline{\frac{1-\overline{x}y}{|1-\overline{x}y|^8}}= \frac{1-\overline{y}x}{|1-\overline{y}x|^8} = S_{S^7}(y,x).
$$
\item
In the case of the octonionic half-space we have
$$
\overline{S_{H^+}(x,y)} = \overline{\frac{\overline{x}+y}{|\overline{x}+y|^8}} = \frac{\overline{y}+x}{|\overline{y}+x|^8} = S_{H^+}(y,x).
$$
\item
Finally, in the case of the strip domain $T$ we have
$$
\overline{S_{T}(x,y)} = \sum\limits_{n=-\infty}^{+\infty} (-1)^n \overline{\frac{\overline{x}+y+2dn}{|\overline{x}+y+2dn|^8}} = \sum\limits_{n=-\infty}^{+\infty}(-1)^n \frac{\overline{y}+x+2dn}{|\overline{y}+x+2dn|^8} = S_{T}(y,x),
$$
\end{itemize}
and this ends the proof.
\end{proof}

\subsection{Bergman spaces of octonionic monogenic functions}

In the case of the Hardy spaces, we consider the integration over the boundary of a domain. From the geometrical point of view it was therefore absolutely natural to consider the normal unit field $n(x)$ inside the definition of the inner product.

An important question is to ask for analogous considerations when integrating over a volume, as in the case of the Bergman space. Here again we encounter the fundamental problem of how to introduce a meaningful definition of an inner product. In contrast to the Hardy space case, the volume integration normally is not related with any boundary directions, so it would be natural to simply consider
$$
(f,g) = \frac{3}{\pi^4} \int\limits_{\Omega} \overline{g(x)}f(x) dV(x),
$$
where $dV(x) = dx_0 \wedge dx_1 \wedge \ldots \wedge dx_7$ is the scalar volume measure. However, here again we would encounter the problem of how to relate such an integral with the octonionic Cauchy integral formula.
Therefore, in the case of the unit ball, in \cite{WL2018} the authors use the following modified inner product
$$
\frac{3}{\pi^4} \int\limits_{B_8} \overline{\left(\frac{x}{|x|} g(x)\right)}\left(\frac{x}{|x|} f(x)\right) dV(x)
$$
while in their-follow up paper they suggest for the half-space case the usual inner product
$$
\frac{3}{\pi^4} \int\limits_{H^+} \overline{g(x)}f(x) dV(x).
$$
In fact their choices allowed them to use the octonionic Cauchy integral formula to prove that their expressions for the Bergman kernel really reproduce all $L^2$-functions with respect to that particular choice of inner product. In fact, using directly the Cauchy formula they prove in \cite{WL2018} the following result:
\begin{proposition} (Bergman kernel for the unit ball).\\
The Bergman kernel of left octonionic monogenic functions on the octonionic unit ball $B_8(0,1)$ is given by
$$
B_{B_8}(x,y) = \frac{6(1-|x|^2y^2)(1-\overline{x}y)}{|1-\overline{x}y|^{10}} + \frac{2(1-\overline{x}y)(1-\overline{x}y)}{|1-\overline{x}y|^{10}}.
$$
\end{proposition}
Moreover, in \cite{WL2020} they prove:
\begin{proposition} (Bergman kernel for the half-space)\\
The Bergman kernel of the Bergman space of left octonionic monogenic functions on the octonionic half-space $H^+(\mathbb{O}) = \{ z \in \mathbb{O} \mid x_0 > 0\}$ is given by
$$
B_{H^+}(x,y) = - 2 \frac{\partial }{\partial x_0}
\left( \frac{\overline{x}+y}{|\overline{x}+y|^8} \right).
$$
\end{proposition}
For the sake of completeness we also recall the formula for the strip domain $T$ proved in \cite{KraMMAS}:
\begin{proposition} (Bergman kernel for the strip domain $T$)\\
Let $d > 0$. Then the Bergman kernel of the Bergman space of left octonionic monogenic functions on the octonionic strip domain $T = \{x \in \mathbb{O} \mid 0 < x_0 < d\}$ is given by
$$
B_{T}(x,y) = (- 2) \sum\limits_{n=-\infty}^{+\infty}\frac{\partial }{\partial x_0}
\left( \frac{\overline{x}+y+2dn}{|\overline{x}+y+2dn|^8} \right).
$$
\end{proposition}
However, in \cite{WL2018,WL2020} there was no hint how to address more general domains and what are the requirements to guarantee the existence of a unique Bergman kernel. Here we encounter the same kinds of problems as in the case of the Hardy spaces discussed in the previous section.
\par\medskip\par
Although the use of a directional entity in the definition of an inner product for the Bergman space might look a bit unnatural at the first glance, it may be justified by the preceding results. Recall that we are generally assuming that $\Omega \subset \mathbb{O}$ is a simply connected orientable domain with strongly Lipschitz boundary $\partial \Omega$, where the outward directed unit normal field denoted by $n(y)$ exists at almost every point $y$ of $\partial \Omega$.

\par\medskip\par

Now let $\omega(x)$ be a weight factor satisfying $|\omega(x)| = 1$ almost everywhere in $\Omega$ and $\omega(x)=0$ at most in a Lebesgue-null set of $\Omega$. Then define
$$
(f,g)_{\Omega,\omega} = \int\limits_{\Omega}(\overline{\omega(x) g(x)}) (\omega(x) f(x)) dV(x)
$$
which will be abbreviated by $(f,g)_{\Omega}$ when it is clear which concrete weight factor is considered.
\begin{definition}
Let $\Omega$ be a convex domain with a smooth oriented boundary $\partial \Omega$. If $x$ is an interior point of $\Omega$ such that there is a uniquely defined point $\tilde{x} \in \partial \Omega$ satisfying
$$
|x-\tilde{x}| < |x-y| \;\;\forall y \in \partial \Omega \backslash\{\tilde{x}\},
$$
then we define the weight factor as $\omega(x):=n(\tilde{x})$, where $n(\tilde{x})$ indicates the exterior normal unit vector at $\tilde{x} \in \partial \Omega$.
For all other interior points $x \in \Omega$ that do not have a uniquely defined nearest point in the boundary we put $\omega(x):=0$.
\end{definition}
{\bf Assumption}. Let us suppose for the rest of this section that $\Omega$ is  a convex domain with a smooth oriented boundary $\partial \Omega$. Moreover we assume that the set $\Omega$ is such that the Lebesgue measure of all points $x \in \Omega$ which do not have a uniquely defined nearest point on the boundary is zero.
\begin{example}
In the concrete example of the unit ball $B_8(0,1)$ we have
$$
\omega(x) = \left\{ \begin{array}{cc} \frac{x}{|x|} & x \neq 0\\ 0 & x=0 \end{array} \right.
$$
Here, in fact $n(\tilde{x}) = \frac{x}{|x|}$ for all $x \neq 0$.
\par\medskip\par
In the case of the right octonionic half-space $H^+(\mathbb{O})$ we have $\omega(x) \equiv 1$ for all $x \in H^{+}(\mathbb{O})$.
\par\medskip\par
In the case of the strip domain $T = \{x \in \mathbb{O} \mid 0 < {\rm Re}(x) < d\}$ we have
$$
\omega(x) = \left\{ \begin{array}{cc} -1 & 0 < {\rm Re}(x) < \frac{d}{2}\\
	0 & {\rm Re}(x) = \frac{d}{2}\\ 1 & \frac{d}{2} < {\rm Re}(x) < d \end{array}
	\right.
$$
Note that the set of $x \in T$ with ${\rm Re}(x) = \frac{d}{2}$ has in fact Lebesgue measure zero.
 \end{example}
The following definition thus makes sense and encompasses all the previously introduced particular definitions for $B_8(0,1), H^+(\mathbb{O})$ and $T$. It also includes the classical case $\omega(x) \equiv 1$. Note that in the case of associativity these weight factors will be cancelled out, since $\overline{\omega(x)} \omega(x) = 1$.

\begin{definition}\label{innerBergman}
Let $\Omega \subset \mathbb{O}$ be a domain meeting the special requirements described above and let $\omega(x)$ be as above.
For octonion-valued functions $f,g \in L^2(\Omega)$ one defines the following $\mathbb{R}$-linear $\mathbb{O}$-valued inner product
$$
(f,g)_{\Omega} := \frac{3}{\pi^4} \int\limits_{\Omega}(\overline{\omega(x)g(x)})\; (\omega(x)f(x)) dV(x)
$$
The space of octonion-valued functions
$$
{\cal{B}}^2(\Omega) := L^2(\Omega) \cap \{f: \Omega \to \mathbb{O} | {\cal{D}}f(x) =  0\;\;\forall x \in \Omega \}
$$
is called the Bergman space of left octonionic monogenic functions.
\end{definition}
This inner product induces the usual $L^2(\Omega)$-norm, i.e.
$$
(f,f)_{\Omega} = \frac{3}{\pi^4} \int\limits_{\Omega}(\overline{\omega(x)f(x)})\; (\omega(x)f(x)) dV(x) = \frac{3}{\pi^4} \int\limits_{\Omega} |f(x)|^2 dV(x),
$$
were we again make use of Artin's theorem. We can say more. Using the same chain of arguments as in the proof of Proposition~\ref{criteriaHardy} just performing the integration over the full domain instead of performing it over the boundary allows us to establish its complete analogue for the Bergman space of octonionic monogenic functions.

\begin{proposition}
The set $({\cal{B}}^2(\Omega),(\cdot,\cdot)_{\Omega})$ is an octonionic Hilbert space in the sense of Definition~\ref{defHilbert}.
\end{proposition}

Note of course, that in order to apply Definition~\ref{innerBergman}, we may only consider domains $\Omega$ where the unit normal field $n(x)$ exists at least almost everywhere on the boundary of $\Omega$. This is a condition that is normally not imposed in the treatment of Bergman spaces.  On the other hand from the background of Stokes' theorem the use of the normal field is not completely unmotivated. Again, in the associative case, the normal field (weight factor) $\omega(x)$ is cancelled out in these representations. As already mentioned, in the case of the unit ball we get exactly $n(x)=\frac{x}{|x|}$ and we re-discover the definition proposed in \cite{WL2018} for the unit ball. Also in the half-space case where we have $n(x)=-1$ we re-obtain the definition used in \cite{WL2020}. Our considerations are compatible with the recent research results presented in this direction.

Again we have to stress the fact that the octonionic valued inner product $(f,g)_{\Omega}$ is only $\mathbb{R}$-linear and not $\mathbb{O}$-linear, as one may easily verify with some simple counterexamples. It is again $\mathbb{O}$-para-linear in the sense of \cite{QR2021,QR2022} such as more explicitly described in the previous subsection addressing the Hardy space case.

Nevertheless, we can still prove that

\begin{proposition}
The Bergman space of left octonionic monogenic functions ${\cal{B}}^2(\Omega)$ satisfies the classical Bergman condition.
\end{proposition}

\begin{proof}
Let us consider a closed ball $B_8(x,r)$ centered at $x$ with radius $r$ such that it is completely contained in the inside of the domain $\Omega$ that we consider. In \cite{XL2002} an octonionic mean value theorem is proved for octonionic monogenic functions, stating that we have the following representation
$$
f(x) = \frac{1}{r^8V_8} \int\limits_{B_8(x,r)} f(y) dV(y),
$$
where $V_8 = \frac{\pi^4}{24}$ is the volume of the eight-dimensional unit ball. In fact, this mean value theorem is a direct consequence of the octonionic Cauchy integral formula, which again serves as main argument here. Rewriting the octonionic function $f$ in terms of its eight real component functions, say $f(y) = \sum\limits_{i=0}^7 e_i f_i(y)$ where each $f_i(y)$ is real-valued. Then the mean value property can be expressed in the form
$$
f(x) = \frac{1}{r^8 V_8} \sum\limits_{i=0}^7 e_i \int\limits_{B(x,r)} f_i(y) dV(y).
$$
On the level of the real component functions we may now apply the usual real Cauchy-Schwarz inequality which yields
\begin{eqnarray*}
|f(x)|^2 &=& \frac{1}{r^{16} V_8^2} \sum\limits_{i=0}^7 (\int\limits_{B(x,r)} f_i(y) dV(y))^2 \\
&\le & \frac{1}{r^8 V_8} \sum\limits_{i=0}^7 \int\limits_{B(x,r)}(f_i(y))^2 dV(y)\\
&=& \frac{1}{r^8V_8} \int\limits_{B(x,r)} \overline{(n(y)f(y))}(n(y)f(y))dV(y)\\
&=& \frac{1}{r^8V_8} \frac{\pi^4}{3} \|f\|^2_{B(x,r)} \\
&\le& \frac{1}{r^8 V_8} \frac{\pi^4}{3} \|f\|^2_{\Omega} \\
&=& \frac{8}{r^8}\|f\|^2_{\Omega},
\end{eqnarray*}
where $\|f\|^2_{\Omega} = (f,f)_{\Omega} = \frac{3}{\pi^4} \int\limits_{\Omega}(\overline{\omega(y) f(y)})(\omega(y)f(y))dV(y)$. We just have proven that ${\cal{B}}^2(\Omega)$ has a continuous point evaluation.
\end{proof}
However, in view of the lack of the $\mathbb{O}$-linearity, of a Cauchy-Schwarz inequality and of the Fischer-Riesz representation theorem, to get the existence of a uniquely defined Bergman kernel we have again to consider analogously to the previous subsection the real-valued inner products. Concretely,
$$
\langle f, g \rangle _0 := {\rm Re}\{(f,g)_{\Omega} \}.
$$
Also in this context we have $\langle \omega f, \omega g\rangle_0 = \langle f,f\rangle_0$ and therefore $$\langle f,g\rangle_0 = \frac{3}{\pi^4}\int\limits_{\Omega} {\rm Re} \{\overline{g(x)} f(x)\} dV(x).$$ Note that the $\mathbb{O}$-para-linearity condition $\langle f \alpha,g\rangle_0 = {\rm Re}((f,g)\alpha)$ (Definition \ref{defHilbert} (vi)) is satisfied.

For the other $i=1,\ldots,7$ we also define $\langle f,g\rangle_i := \{(f,g)_{\Omega}\}_i$ denoting the $i$-th real component of $(f,g)_{\Omega}$.

Then again we can reproduce each of the eight scalar components separately and define a partial Bergman projection for each component by
${B_i}_x: y \mapsto  {B_i}_x(y) = {B_i}(x,y)$. Then we have $[{\cal{B}}_i f](x) := \langle f, \rangle {B_i}_x\rangle_i = f_i(x)$ for each $i=0,\ldots,7$.
Again then we can define a total Bergman projection by
$$
[{\cal{B}} f] := \sum\limits_{i=0}^7 \langle f , B_i \rangle_i e_i.
$$
If $f$ is a function from $B^2(\Omega)$, then we then have per construction:
$$
[{\cal{B}} f](x) = \sum\limits_{i=0}^7 f_i(x) e_i = f(x),
$$
which is the reproduction property.

Again, the deal then is to look for a global kernel function $B(x,y)$ that allows us to write this total Bergman projection in the form
$$
[{\cal{B}} f](x) = \frac{3}{\pi^4}\int\limits_{\Omega} (\overline{\omega(y) B(x,y)}) (\omega(y) f(y)) dV(y)
$$
involving one global octonionic function function $B(x,y)$.

In the particular cases of the unit ball, the half-space and the strip domain $T$ we are sure to have such global representations.

Finally we can also show that
\begin{proposition}
In the case of the octonionic unit ball $B_8(0,1)$, the octonionic half-space $H^+(\mathbb{O})$ and the strip domain $T$ the total Bergman projection is self-adjoint and hence an orthogonal projector with respect to $\langle\cdot, \cdot\rangle_0$.
\end{proposition}
\begin{proof}
To the proof we again can show by a simple calculation that $\overline{B_{B_8}(y,x)}=B_{B_8}(x,y)$.  In the case of the unit ball we get by a computation that:
$$
\overline{B_{B_8}(x,y)} = \frac{6(1-|y|^2|x|^2)(1-\overline{y}x)}{|1-\overline{x}y|^{10}} + \frac{2(1-\overline{y}x)(1-\overline{y}x)}{|1-\overline{x}y|^{10}}
$$
$$
\frac{6(1-|y|^2|x|^2)(1-\overline{y}x)}{|1-\overline{y}x|^{10}} + \frac{2(1-\overline{y}x)(1-\overline{y}x)}{|1-\overline{y}x|^{10}}
= B_{B_8}(y,x).
$$
Let us now turn to the case of the half-space.

Here we start from the formula $B_{H^+}(x,y) = - 2 \frac{\partial }{\partial x_0}
\Bigg\{ \dfrac{\overline{x}+y}{|\overline{x}+y|^8} \Bigg\}$ which can be rewritten as
$$
B_{H^+}(x,y) = - 2 \Bigg(
\frac{|\overline{x}+y|^8 - (\overline{x}+y) 8 |\overline{x}+y|^3(x_0+y_0)}{|\overline{x}+y|^{16}}
\Bigg).
$$
One directly sees that $\overline{B_{H^+}(x,y)}=B_{H^+}(y,x)$ and similarly we may conclude that $\overline{B_{T}(x,y)}=B_{T}(y,x)$.

So in these three cases we have $\overline{B(x,y)}=B(y,x)$ and consequently the Bergman projection is orthogonal and self-adjoint with respect to $\langle, \rangle_0$ in these cases.
\end{proof}

\section{The octonionic slice monogenic setting}

In this section we introduce Hardy and Bergman spaces of octonionic slice monogenic functions according to the classical book structure as defined in Section 2.3. Also in this section, by space we a priori only mean a linear space over $\mathbb{R}$. However, we shall see that the spaces that we are going to introduce are endowed with the $\mathbb{O}$-para-linear structure like their monogenic counterparts that we previously described.

\subsection{General definition of Hardy spaces of slice monogenic functions}

Let $\Omega\subset\mathbb O$ be such that $\Omega\cap\mathbb C_I$ is an axially symmetric simply connected, orientable domain with strongly Lipschitz boundary.
\begin{definition}\label{hardyslice}
We define the Hardy space $\mathbf{H}^2(\Omega,\mathbb O)$ as the closure of the set of slice monogenic functions on $\Omega$ belonging to $L^2(\partial\Omega\cap\mathbb C_I)$ for all $I\in\mathbb S$, namely
$$
\int_{\partial\Omega\cap\mathbb C_I} |f(z)|^2 |ds(z)| <\infty,
$$
where $|ds(z)|$ is the scalar-valued arc element as precisely defined below.
\end{definition}
As we shall describe below, also in the slice monogenic setting it is necessary to use a suitable inner product which takes into account a field defined on the boundary of $\Omega$. In this case, it is the normalized, $\mathbb C_I$-valued, tangential field $t(z)$ on the curve $\partial\Omega\cap\mathbb C_I$, where $I\in\mathbb S$. The need of a field to define the inner product is a peculiarity of the non associative setting. The arc element on the curve can be written as $t(z)|ds(z)|$ where $|t(z)ds(z)|=|ds(z)|$ and $|t(z)|^2=1=\overline{t(z)}t(z)$.
\\
We equip $\mathbf{H}^2(\Omega,\mathbb O)$ with the $\mathbb O$-valued $\mathbb R$-bilinear form
\begin{equation}\label{innerp}
[f,g]_{I}:=\int_{\partial\Omega\cap\mathbb C_I} \overline{(t(z)g(z))}\, |ds(z)|\, (t(z) f(z)), \qquad z=u+Iv
\end{equation}
for some fixed, but arbitrary $I\in\mathbb S$. This bilinear form depends on $\Omega$ and also on the choice of $I$ in $\mathbb S$. Later we shall use the notation $\{\cdot\}_i$ to denote the $i$-part, and we can write
$$
[f,g]_{I}=\sum_{i=0}^7 \{[f,g]_{I}\}_i e_i
$$
where, in particular, $\{[f,g]_{I}\}_0={\rm Re} [f,g]_{I}$.
\\
We now prove that

\begin{proposition}\label{criteriaHardy1}
The set $(\mathbf{H}^2( \Omega, \mathbb{O}),[\cdot,\cdot]_{I})$ is an octonionic Hilbert space in the sense of Definition~\ref{defHilbert}.
\end{proposition}
\begin{proof}
The properties (i), (iv) are immediate. We then prove (ii), (iii), (v), (vi).
\begin{itemize}
\item[(ii)] To prove that $[\cdot,\cdot]_{\Omega,I}$ is hermitian we compute
$$
\overline{[f,g]_{I}}=\overline{\int_{\partial\Omega\cap\mathbb C_I} \overline{(t(z)g(z))}\, |ds(z)|\,(t(z) f(z))}=\int_{\partial\Omega\cap\mathbb C_I} \overline{(t(z)f(z))}\, |ds(z)|\, (t(z) g(z))=[g,f]_I.
$$
\item[(iii)] We have
$$[f,f]_{\Omega,I}   = \int_{\partial\Omega\cap\mathbb C_I} \overline{(t(z)f(z))}\, |ds(z)|\, (t(z)f(z))=\int_{\partial\Omega\cap\mathbb C_I} |t(z)f(z)|^2\, |ds(z)|
$$
$$
=\int_{\partial\Omega\cap\mathbb C_I} |f(z)|^2\, |ds(z)|\geq 0$$
 so in particular $[f,f]_I=0$ if and only if $f=0$ because $f$ needs to be continuous in particular as an element belonging to $H^2$.
\item[(iv)]  Since each real number commutes with any octonion, this property follows directly.
\item[(v)] This is a consequence of Artin's theorem. The expression under the integral is generated by two octonionic elements, namely by $\alpha$ and by $F(z)$ where $F(z):= t(z) f(z)$.
\item[(vi)]  We follow the argument in the proof of Proposition \ref{criteriaHardy} making use of Proposition 1.4 (e) in \cite{dieckmann}, see \eqref{1.4e}:
\[
\begin{split}
{\rm Re}[ f \alpha, g ] &= \int\limits_{\partial \Omega\cap\mathbb C_I} {\rm Re}\left(\overline{(t(z)g(z))} \, |ds(z)|\,(t(z)f(z)\alpha) \right)\\
&= \int\limits_{\partial \Omega\cap\mathbb C_I} {\rm Re}\left(\overline{(t(z)g(z))} \, |ds(z)|\,((t(z)f(z))\alpha) \right)
\\
&= \int\limits_{\partial \Omega\cap\mathbb C_I} {\rm Re}\left(\overline{(t(z)g(z))}  \, |ds(z)|\,(t(z)f(z))\alpha \right)= {\rm Re}([ f , g ]\alpha).
\end{split}
\]
and this is the required para-linearity property.
\end{itemize}
\end{proof}
Below we study the specific case of the unit ball and of the half space. In the first case, we show an alternative way to introduce an inner product which turns out to be equivalent to \eqref{innerp}.

\subsection{Hardy spaces of octonionic slice monogenic functions in the unit ball}


We start with first looking at the particular case of slice monogenic functions in the octonionic unit ball $B_8(0,1)$. In this particular case, they can be written as convergent power series centered at the origin and, as in the classic complex case, we shall exploit this feature to define an inner product. If $f,g\in \mathcal{SM}(B_8(0,1))$ and $f(x)=\sum_{n \geq 0} x^na_n$,
$g(x)=\sum_{n \geq 0} x^nb_n$ we define the $\mathbb O$-valued inner product
\begin{equation}\label{innerslice}
[f,g] :=\sum_{n\geq 0} \overline{b_n}a_n.
\end{equation}
First we observe
\begin{proposition}\label{Prop1}
The inner product \eqref{innerslice} is $\mathbb R$-linear and defines the norm
\begin{equation}\label{normslice}
\|f\| = [f,f]^{1/2}=\left(\sum_{n\geq 0}|a_n|^2\right)^{1/2}.
\end{equation}
\end{proposition}
\begin{proof}
The fact that the inner product \eqref{innerslice} is additive and $\mathbb R$-linear follows by the additivity and linearity of the series \eqref{innerslice}. It is immediate that \eqref{normslice} defines a norm.
\end{proof}

\begin{definition}
We define the octonionic Hardy space of the unit ball, and we denote it by  $\mathbf{H}^2(B_8(0,1),\mathbb O)$, the subset of $\mathcal{SM}(B_8(0,1))$ consisting of functions $f(x)=\sum_{n \geq 0} x^na_n$ for which $\sum_{n\geq 0}|a_n|^2<\infty$ and equipped with the inner product \eqref{innerslice}.
\end{definition}
\begin{proposition}
$(\mathbf{H}^2(B_8(0,1),\mathbb O),[\cdot,\cdot])$ is an octonionic Hilbert space in the sense of Definition~\ref{defHilbert}.
\end{proposition}
\begin{proof}
Since (i) and (iv) are trivial and Proposition~\ref{Prop1} holds, it remains to verify the properties (ii) and (v), (vi) of Definition~\ref{defHilbert}. We consider $f(x)=\sum_{n\geq 0} x^na_n$,
$g(x)=\sum_{n\geq 0} x^nb_n$ in $\mathbf{H}^2(B_8(0,1),\mathbb O)$.
\begin{itemize}
\item[(ii)] Since $[g,f]= \sum_{n\geq 0}\overline{a_n}b_n$ it is clear that $\overline{[g,f]}=\overline{\sum_{n\geq 0}\overline{a_n}b_n}= \sum_{n\geq 0}\overline{b_n}a_n=[f,g].$
\item[(v)]Let $\alpha \in \mathbb{O}$. Then
$$
[f \alpha,f] = \sum_{n\geq 0}\overline{a_n}(a_n\alpha)= \sum_{n\geq 0}(\overline{a_n}a_n)\alpha=[f,f]\alpha,
$$
where we used left alternativity.
\item[(vi)] This property ($\mathbb{O}$-para-linearity) follows from the already quoted Proposition 1.4 (e) in \cite{dieckmann} recalled in \eqref{1.4e}. In fact, for $\alpha\in\mathbb O$ we have
$$
{\rm Re} \left([f\alpha,g]\right)={\rm Re} \left(\sum_{n\geq 0}\overline{b_n}(a_n\alpha)\right)=
\left(\sum_{n\geq 0}{\rm Re}(\overline{b_n}(a_n\alpha))\right)=\left(\sum_{n\geq 0}{\rm Re}((\overline{b_n}a_n)\alpha)\right)
$$
$$
={\rm Re}\left(\sum_{n\geq 0}(\overline{b_n}a_n)\alpha\right)= {\rm Re}([f,g]\alpha).
$$
\end{itemize}
\end{proof}

\par\medskip\par
Next we provide the counterpart of the Szeg\"o kernel in this framework motivated by the quaternionic case. Let us consider
\begin{equation}\label{kesse}
{\mathcal S}(y,x)= (1-2{\rm {\rm Re}}(x)y +|x|^2y^2)^{-1}(1-yx)=(1-\bar y \bar x)(1-2{\rm {\rm Re}}(y)\bar x +|y|^2\bar x^2)^{-1}.
\end{equation}
This kernel is well-defined for $1-2{\rm {\rm Re}}(x)y +|x|^2y^2 \not=0$, namely for $x\not\in [y^{-1}]$,  and it is immediate that it is octonionic Hermitian, namely:
$$
\overline{{\mathcal S}(y,x)}={\mathcal S}(x,y).
$$
The kernel is called
the octonionic slice monogenic Szeg\"o kernel. It is
left slice monogenic in $y$ and right slice monogenic in $\bar x$.
Note also that for $y,x\in B_8(0,1)$ we have ${\mathcal S}(y,x)=\sum_{n\geq 0} y^n \bar x^n$ and that ${\mathcal S}(\cdot,x)\in \mathbf{H}^2(B_8(0,1),\mathbb O)$ in fact $|{\mathcal S}(y,x)|^2 \le \sum_{n\geq 0} |x|^{2n}<\infty$.
Like in the monogenic case we can prove:
\begin{theorem}
The space $\mathbf{H}^2(B_8(0,1),\mathbb O)$ has a uniquely defined reproducing kernel having the reproducing property:
$$
[ f, {\mathcal S}(\cdot,x)]=f(x),
$$
for any $f\in\mathbf{H}^2(B_8(0,1),\mathbb O)$.
\end{theorem}
\begin{proof}
Let $f(x)=\sum_{n\geq 0} x^na_n$.
We immediately have the following equalities
$$
[ f(y), {\mathcal S}(y,x)]=\left[\sum_{n\geq 0} y^n a_n, \sum_{n\geq 0} y^n \bar x^n\right]=\sum_{n\geq 0} x^n a_n=f(x),
$$
and the statement follows.
\end{proof}
We now go back to the description at the beginning of this section, namely to the analogue of a geometric characterization of the Hardy space, as we did in the monogenic case in Section 3.
We consider functions $\mathcal{SM}(B_8(0,1)$ and adapt \eqref{innerp} to this case, so that $z=e^{I\theta}$, $t(z)=Ie^{I\theta}$, $|ds(z)|= d\theta$, $\theta\in [0,2\pi]$:
\begin{equation}\label{L2}
[f,f]_I:= \frac{1}{2\pi}\int_0^{2\pi} \left(\overline{Ie^{I\theta} f(e^{I\theta})}) d\theta (Ie^{I\theta}f(e^{I\theta})) \right)<\infty
\end{equation}
for some $I\in\mathbb S$. We have put the constant $1/2\pi$ in front in order to have $[1,1]_I=1$.
\\ We have:
\begin{proposition}\label{prop47}
(i) The function $f\in\mathbf{H}^2(B_8(0,1),\mathbb O)$ if and only if $[f,f]_I$ is finite for some $I\in\mathbb S$.
\\
(ii) Let $I,J\in\mathbb S$. Then $[f,f]_I$ is finite if and only if $[f,f]_J$ is finite.
\end{proposition}
\begin{proof}
To show (i) we write $f\in\mathcal{SM}(B_8(0,1))$ as $f(x)=\sum_{n\geq 0} x^na_n$ and using repeatedly \eqref{1.4e} we write the integrand in \eqref{L2} (where we write, for simplicity $z$ instead of $e^{I\theta}$) as
\[
\begin{split}
&{\rm Re}\left(\left(\overline{t(z)(\sum_{n\geq 0} z^na_n})\right)|ds(z)| \left(t(z)(\sum_{n\geq 0} z^na_n)\right) \right)\\ &=
{\rm Re}\left(\left((\sum_{n\geq 0} \overline{a_n} {\bar z}^n) \overline{t(z)}\right)|ds(z)| \left(t(z)\sum_{n\geq 0} z^na_n\right) \right)\\
&={\rm Re}\left((\sum_{n\geq 0} \overline{a_n} {\bar z}^n\, ) \left(\overline{t(z)}|ds(z)| \left(t(z)(\sum_{n\geq 0} z^na_n)\right)\right)\right)={\rm Re}\left((\sum_{n\geq 0} \overline{a_n} {\bar z}^n \,) \overline{t(z)}|ds(z)| t(z) (\sum_{n\geq 0} z^na_n) \right)\\
&={\rm Re}\left((\sum_{n\geq 0} \overline{a_n} {\bar z}^n \,) |ds(z)| (\sum_{n\geq 0} z^na_n) \right)
={\rm Re}\left(\sum_{n,m\geq 0} \overline{a_n} {\bar z}^n \, |ds(z)|\, z^m\, a_m) \right)\\
&={\rm Re}\left(\sum_{n,m\geq 0} \overline{a_n} ({\bar z}^n |ds(z)|\, z^m)\, a_m)  \right).\\
\end{split}
\]
Since, by Proposition \ref{criteriaHardy1}, $[f,f]_I$ is real, using the previous calculuations we have that \eqref{L2} rewrites as
\[
[f,f]_I=\frac{1}{2\pi}\int_0^{2\pi} \overline{(Ie^{I\theta}f(e^{I\theta}))}\, d\theta\, (Ie^{I\theta}f(e^{I\theta}))=\frac{1}{2\pi}{\rm Re}\left(\int_0^{2\pi}\overline{(Ie^{I\theta}f(e^{I\theta}))}\, d\theta\, (Ie^{I\theta}f(e^{I\theta}))\right)
$$
$$
=\frac{1}{2\pi}\int_0^{2\pi} {\rm Re}\left(\overline{(Ie^{I\theta}f(e^{I\theta}))}\, (Ie^{I\theta}f(e^{I\theta}))\, d\theta\right)
= \frac{1}{2\pi}\int_0^{2\pi}{\rm Re}\left(\sum_{n,m\geq 0} \overline{a_n} (e^{-In\theta} \, e^{Im\theta})\, a_m)  \right)\, d\theta.
\]
Since
$$
\int_0^{2\pi} e^{-In\theta}\, e^{Im\theta}\, d\theta=2\pi\delta_{n,m}
$$
we conclude that
$$
[f,f]_I=\sum_{n\geq 0}|a_n|^2
$$
and the statement follows. This also proves the validity of (2).
\end{proof}
From this result we now obtain:
\begin{corollary}
The Hardy space $\mathbf{H}^2(B_8(0,1),\mathbb O)$ consists of the subset of $\mathcal{SM}(B_8(0,1))$ such that $[f,f]_I$ is finite for some $I$ (and so for all) in $\mathbb{S}$.
\end{corollary}

\begin{remark}\label{reprodpro} We note that to obtain a description of the kernel according to the inner product $[\cdot,\cdot]_I$ we consider $x,y\in\mathbb C_I$ and $[ f(y), {\mathcal S}(y,x)]_I$ where ${\mathcal S}(y,x)$ is defined in \eqref{kesse} and the  inner product is \eqref{innerp}.
We have:
\[
\begin{split}
[ f(y), {\mathcal S}(y,x)]_I&=\frac{1}{2\pi}\int_{S^7\cap\mathbb C_I} \overline{(t(y) {\mathcal S}(y,x) )}\, |ds(z)|\, (t(y) f(y))\\
&= \frac{1}{2\pi}\int_{S^7\cap\mathbb C_I} (\overline{{\mathcal S}(y,x)}\, \, \overline{t(y)}  )\, |ds(z)|\, (t(y) f(y))
\end{split}
\]
and using Artin's theorem we obtain that the previous expression equals:
\[
\begin{split}
&=\frac{1}{2\pi}\int_{S^7\cap\mathbb C_I} {\mathcal S}(x,y)\, |ds(z)|\, f(y)\\
&=\frac{1}{2\pi}\int_{S^7\cap\mathbb C_I} {\mathcal S}(x,y)\, |ds(z)|\,(F_1(y)+F_2(y)I_2+(G_1(y)+G_2(y)I_2)I_4)  =f(x)
\end{split}
\]
where we used the Splitting Lemma to write $f(y)$, again Artin's theorem and applying the reproducing property to the components of $f$. The reproducing property of kernel holds on each complex plane. The values of the function $f$ can be computed at any point in the unit ball using the Representation Formula.
\end{remark}
\subsection{Hardy space of slice monogenic functions of the half-space}
Let us now consider the octonionic half-space
$$H^+(\mathbb O) =\{ x\in\mathbb O\ |\
{\rm Re}(x)>0\}$$ and set $H_I^+(\mathbb O)= H^+(\mathbb O)\cap
\mathbb{C}_I$. We will denote by $f_I$ the restriction of a function $f$ defined on $H^+(\mathbb O)$ to $H_I^+(\mathbb O)$. According to Definition \ref{hardyslice}, we set
\[
\mathbf{H}_2(H_I^+(\mathbb O))=\{f\ {\rm slice\ monogenic\ in\ H^+(\mathbb O)}\ : \
\int_{-\infty}^{+\infty}   |f_I(Iv)|^2 dv <\infty\},
\]
where the variable in $\mathbb C_I$ is denoted by $u+Iv$ and $f(Iv)$ denotes the non-tangential value of $f$ at $Iv$. Note that these values  exist almost everywhere, in fact any  $f\in \mathbf{H}_2(H_I^+(\mathbb O))$ when restricted to a complex plane $\mathbb C_I$ can be written, by the Splitting lemma, as $f(z)=F_1(z)+F_2(z)I_2+(G_1(z)+G_2(z)I_2)I_4$ for suitable
$I_2$ and $I_4$ in $\mathbb{S}$, with
$F_1,F_2,G_1,G_2$ holomorphic functions from $\Omega\cap \mathbb{C}_{I}$.
Since the non-tangential values of $F_1, F_2$ and $G_1,G_2$ exist almost everywhere at $Iv$, also the non-tangential value of $f$ exists at $Iv$ a. e. on $H_I^+(\mathbb O)$ and $f_I(Iv)=F_1(Iv)+F_2(Iv)I_2+(G_1(Iv)+G_2(Iv)I_2)I_4$ almost everywhere.

The inner product \eqref{innerp} in $\mathbf{H}_2(H_I^+(\mathbb O))$ can be rewritten as
\begin{equation}\label{H2O+}
[f,g]_{I}=\int_{-\infty}^{+\infty}
\overline{(t(Iv)g_I(Iv))}(t(Iv) f_I(Iv)) dv=\int_{-\infty}^{+\infty}
\overline{(Ig_I(Iv))} (If_I(Iv)) dv ,
\end{equation}
where $t(Iv)=I$ and $f_I(Iv)$, $g_I(Iv)$ denote the nontangential values of $f,g$
at $Iv$ on $H_I^+(\mathbb O)$.
\\
By Proposition \ref{criteriaHardy1}, $(\mathbf{H}_2(H_I^+(\mathbb O)), [
\cdot,\cdot]_{I}$is an octonionic Hilbert space.
This scalar product gives the norm
\[
\|f\|_{\mathbf{H}_2(H_I^+(\mathbb O))}=\left(\int_{-\infty}^{+\infty}  |
f_I(Iv)|^2 dv\right)^{\frac 12},
\]
which is finite by our assumptions. By the Representation Formula, for any $I,J\in\mathbb S$, $[
f,f]_{I}$ is finite if and only if $[
f,f]_{J}$ is finite and this gives meaning to the following:
\begin{definition}
The Hardy space of the half-space $\mathbf{H}_2(H^+(\mathbb O))$ is the set of slice monogenic functions on $H^+(\mathbb O)$ such that $[f,f]_{I}$ is finite for some $I\in\mathbb S$.
\end{definition}

The Hardy space $\mathbf{H}_2(H^+(\mathbb O))$ has a reproducing kernel, which is the slice monogenic extension of the kernel of the Hardy space of the right half space, namely:
\begin{proposition}
The function
\begin{equation}\label{kernel}
k(x,y)=\frac{1}{2\pi}(\bar x +\bar y)(|x|^2 +2{\rm {\rm Re}}(x) \bar y
+\bar y^2)^{-1}
\end{equation}
is slice hyperholomorphic in $x$ and $\bar y$ on the left and on
the right, respectively in its domain of definition, i.e. for $x\not\in[\bar y]$.
The restriction of $2\pi k(x,y)$ to $\mathbb{C}_I\times
\mathbb{C}_I$ coincides with $(x+\bar y)^{-1}$. Moreover $k(x,y)$ can be written as:
\begin{equation}\label{kernel1}
k(x,y)=\frac{1}{2\pi}(|y|^2 +2{\rm {\rm Re}}(y) x + x^2)^{-1}( x + y).
\end{equation}
\end{proposition}
\begin{proof} The proof follows exactly the same lines of the proof in the quaternionic case, see
\cite{ACS}.
\end{proof}
\begin{remark}\label{rmk412}
As in the case of the unit ball, see Remark \ref{reprodpro},
the kernel $k(x,y)$ is reproducing on each plane $\mathbb C_I$ in fact we have, using the Splitting Lemma:
$$[f_I(y),k(x,y)]_I=\int_{-\infty}^{+\infty} \frac{1}{2\pi}(x-Iv)^{-1}\, f_I(Iv) \, dv=
$$
$$
= \int_{-\infty}^{+\infty} \frac{1}{2\pi}(x-Iv)^{-1}\, (F_1(Iv)+F_2(Iv)I_2+G_1(Iv)I_4+G_2(Iv)I_2I_4)\, dv
$$
and since each summand involves two imaginary units at a time, we deduce that the previous term is equal to the expression 
$$
(F_1(x)+F_2(x)I_2+G_1(x)I_4+G_2(x)I_2I_4)=F_I(x), \qquad x\in H_I^+(\mathbb O).
$$
The function $f$ is then reconstructed using the Representation Formula.
\end{remark}

We round off this part by presenting explicit representation formulas for the Szeg\"o and Bergman kernels for the strip domains of the form $T= \{x \in \mathbb{O} \mid 0 < x_0 < d\}$ where $d >0$ is arbitrarily but fixed.In both cases, the description follows from the knowledge of the kernels in the case of the half-space.
By applying a similar periodization (reflection) argument as in the monogenic case we can also deduce an explicit formula for the slice monogenic Szeg\"o kernel of the particular strip domain mentioned above.

\begin{proposition}
Let $d > 0$. The octonionic slice monogenic Szeg\"o kernel of the the strip domain ${\cal{S}} = \{x \in \mathbb{O} \mid 0 < x_0 < d\}$ has the explicit form
$$
K_{T}(y,x) = \sum\limits_{n=-\infty}^{\infty} (-1)^n k(y+2dn,x),
$$
where $k(x,y)$ is as in \eqref{kernel}.
This kernel is left slice monogenic in $y$ and right slice monogenic in $\bar x$ and satisfies the Hermitian property.
\end{proposition}
\begin{proof} The proof of this result is immediate.
\end{proof}
\begin{remark}
The kernel $K_T$ is reproducing on each slice, as it can be proved using the argument in Remark \ref{rmk412} and the fact that when $y,x\in\mathbb C_I$ the kernel is reproducing for functions with values in $\mathbb C_I$.
\end{remark}
\subsection{Bergman spaces of octonionic slice monogenic functions}

In this subsection we define the analogue of the Bergman space in the slice monogenic setting and give an explicit representation of the reproducing Bergman kernel in the context of the unit ball. Moreover, we provide a sequential characterization in analogy to the Hardy space case presented in the previous subsection.
\par\medskip\par
Let us consider an open, bounded, axially symmetric set slice domain $\Omega\subset\mathbb O$ and an arbitrary but fixed $I\in\mathbb O$. Set $\Omega\cap\mathbb C_I$ and define the following space of functions:
$$
L^2(\Omega_I):=\left\{f:\Omega\to\mathbb O\ |\ \int_{\Omega_I} |f_I|^2 d\tilde\sigma<\infty\right\},
$$
where $d\tilde\sigma=du\, dv/{\rm area}(\Omega_I)$ is the Lebesgue measure on $\Omega_I$.
\begin{proposition}
Let us define for any $f,g\in L^2(\Omega_I)$
\begin{equation}\label{L2inner}
<f,g>_I:=\int_{\Omega_I} \bar g_I f_Id\tilde{\sigma}.
\end{equation}
Then \eqref{L2inner} defines an inner product on $L^2(\Omega_I)$ according to (i)-(vi) in Definition \ref{defHilbert}.
\end{proposition}
\begin{proof}
It is evident that \eqref{L2inner} satisfies (i), (ii) and (iv). Also (iii) is satisfied since $<f,f>_I$ is non-negative and $<f,f>_I=0$ if and only if $f=0$ a.e. in $\Omega_I$ and so in $\Omega$. We now show (v):
$$
<f\alpha,f>_I=\int_{\Omega_I} \bar f(f\alpha) d\tilde{\sigma} = \int_{\Omega_I} \bar f f\alpha d\tilde{\sigma} =
<f,f>_I\alpha
$$
where we used Artin's theorem and the fact that $d\tilde{\sigma}$ is real. Finally, we prove (vi):
$$
{\rm Re}(<f\alpha,g>_I)={\rm Re}\left(\int_{\Omega_I} \bar g(f\alpha) d\tilde{\sigma} \right)= {\rm Re}\left(\int_{\Omega_I} \bar g f\alpha d\tilde{\sigma} \right)={\rm Re}(<f,g>_I\alpha)
$$
where we used Proposition 1.4 (e) from \cite{dieckmann} and the fact that $d\tilde{\sigma}$ is real.
\end{proof}
\begin{remark}
Note that in this context we do not require a weight factor of norm $1$ like in the previous cases. The reason is that the appearing differential form $d\tilde{\sigma}$ is scalar valued and that $f$ and $g$ are $\mathbb{C}_I$-valued on each slice $B_8(0,1) \cap \mathbb{C}_I$ and so we can use Artin's theorem. A weight factor of norm $1$ would be canceled out here, so it is not needed here.
\end{remark}
We equip $L^2(\Omega_I)$ with the norm inherited from the inner product, namely
$$
\|f\|_I=\left(\int_{\Omega_I} \bar ff d\tilde{\sigma}\right)^{1/2}.
$$
\begin{remark}
As in the quaternionic case, using the Representation Formula we can prove that for any $I,J\in\mathbb S$
$$
\int_{\Omega_I} |f|^2 d\tilde{\sigma} \leq 2 \int_{\Omega_J} |f|^2 d\tilde{\sigma}
$$
so that the two norms $\|\cdot\|_I$ and $\|\cdot\|_J$ are equivalent. This fact implies that $f\in\mathcal{A}(\Omega_I)$ if and only if $f\in\mathcal{A}(\Omega_J)$.
\end{remark}
In view of the preceding remark, we can now define the octonionic Bergman space on $\Omega$:
\begin{definition}
Consider the set
$$
\mathcal{A}(\Omega_I):= \mathcal{SM}(\Omega)\cap L^2(\Omega_I)
$$
and the set $\mathcal{A} (\Omega)$, called octonionic slice Bergman space, that consists of functions that belong to $\mathcal{A}(\Omega_I)$ for some $I\in\mathbb S$.
\end{definition}
In the specific case $\Omega=B_8(0,1)$ we can make our study more precise. We have the following result
\begin{proposition}\label{p419}
The function $\mathcal{B}(x,y)=\dfrac{1}{\pi}(1-2\bar x \bar y+\bar x^2 \bar y^2 )(1-2 {\rm Re}(x)\bar y+|x|^2\bar y^2)^{-2}$ is a reproducing kernel for the Bergman space $\mathcal{A}(B_8(0,1))$.
\end{proposition}
\begin{proof}
Let us consider $w\in\mathbb O$ and choose $x$ such that $x,y$ belong to the same complex plane $\mathbb C_I$. Under this assumption $\mathcal{B}(x,y)=\dfrac{1}{\pi}\dfrac{1}{(1-x\bar y)^2}$ is the Bergman kernel on $\mathbb C_I$ for $\mathbb C_I$-valued functions. Consider a function $f(x)=\sum_{n\geq 0} x^n a_n$, $a_n\in\mathbb O$ and its restriction to $\mathbb C_I$. Then we have
$$
<f(x), \mathcal{B}(y, x)>_I=\int_{B_8(0,1)\cap\mathbb C_I} \overline{\mathcal{B}(y, x)}f(x) d\tilde{\sigma} =\int_{B_8(0,1)\cap\mathbb C_I} \mathcal{B}(x,y)(\sum_{n\geq 0} x^n a_n) d\tilde{\sigma} .
$$
Using Artin's rule, the uniform convergence of the series of $f(x)$ and the fact that $d\tilde{\sigma}$ is real, we get
$$
<f(y), \mathcal{B}(y, x)>_I=\sum_{n\geq 0}\int_{B_8(0,1)\cap\mathbb C_I} (\mathcal{B}(x,y) y^n) d\tilde{\sigma} a_n =\sum_{n\geq 0} x^n a_n=f(x).
$$
This formula allows to obtain $f(x)$ for $x\in B_8(0,1)\cap\mathbb C_I$. The conclusion follows by using the Representation formula that allows to reconstruct the values of the function $f$ knowing its values on $B_8(0,1)\cap\mathbb C_I$.
\end{proof}
In the case of the unit ball we also have a sequential characterization of the Bergman space.
\begin{proposition}
Let $f(x)=\sum_{n\geq 0} x^na_n$ be convergent in $B_8(0,1)$. Then $f\in\mathcal A(B_8(0,1))$ if and only if $\sum_{n\geq 0} \dfrac{|a_n|^2}{n+1}<\infty$.
\end{proposition}
\begin{proof}
Let us compute
$$
\| f\|_I^2=\int_{B_8(0,1)\cap\mathbb C_I}\overline{\left(\sum_{n\geq 0} a_n z^n\right)} \left(\sum_{n\geq 0} a_n z^n\right)\, d\tilde\sigma.
$$
Since the value is  a real number, and $d\tilde\sigma$ is real we have
$$
\| f\|_I^2=\int_{B_8(0,1)\cap\mathbb C_I}{\rm Re}\left(\overline{\left(\sum_{n\geq 0} a_n z^n\right)} \left(\sum_{n\geq 0} a_n z^n\right)\, d\tilde\sigma\right).
$$
Setting $z=re^{I\theta}$, $0\leq r\leq 1$ and $\theta\in [0,2\pi]$ and
with calculations similar to those in the proof of Proposition \ref{prop47}, we deduce that
$$
\|f\|_I^2
= \frac{1}{\pi}\int_0^1 r^{n+m+1}\, dr\int_0^{2\pi}{\rm Re}\left(\sum_{n,m\geq 0} \overline{a_n} (e^{-In\theta} \, e^{Im\theta})\, a_m)  \right)\, d\theta =\sum_{n\geq 0} \frac{|a_n|^2}{n+1}
$$
and the statement follows.
\end{proof}

We conclude this paper with a proposition providing the Bergman kernel in the case of the half-space and the strip.
\begin{proposition}
  The function $$\mathcal{B}_{H^+}(x,y)=\frac{1}{\pi}(x^2+2{\rm Re}(y) x +|y|^2)^{-2}(x^2+2xy+y^2)$$ is the reproducing kernel for the Bergman space $\mathcal{A}(H^+(\mathbb O))$ of the right half-space.
\\
The function $$\mathcal{B}_T(x,y)=\sum_{n=-\infty}^{+\infty} \frac{1}{\pi}(x^2+2{\rm Re}(2dn+y) x +|2dn+y|^2)^{-2}(x^2+2x(2dn+y)+(2dn+y)^2)$$ is the reproducing kernel for the Bergman space $\mathcal{A}(T)$ of the strip $T$.
\end{proposition}
\begin{proof}
The proof follows by the standard arguments used, e.g., to prove Proposition \ref{p419} and the knowledge of the corresponding kernels in the complex case, see \cite{ConKraCV}.
\end{proof}

{\bf Declarations}: The authors confirm that there are no conflicts of interests. No data sets have been used. There were no fundings received.

\end{document}